\documentclass{amsart}

\newtheorem{theorem}{Theorem}[section]
\newtheorem{conjecture}[theorem]{Conjecture}
\newtheorem{lemma}[theorem]{Lemma}

\theoremstyle{definition}
\newtheorem{definition}[theorem]{Definition}
\newtheorem{construction}[theorem]{Construction}

\theoremstyle{remark}
\newtheorem{remark}[theorem]{Remark}

\numberwithin{equation}{section}

\usepackage{rotating}
\usepackage[latin1]{inputenc}
\usepackage{bbm}
\usepackage{graphicx}
\usepackage{dsfont}
\usepackage{color}
\usepackage[hyperref]{hyperref}

\newcommand{\C}{{\mathbb{C}}}
\newcommand{\R}{{\mathbb{R}}}
\newcommand{\Q}{{\mathbb{Q}}}
\newcommand{\Z}{{\mathbb{Z}}}

\renewcommand{\epsilon}{\varepsilon}
\renewcommand{\theta}{\vartheta}

\renewcommand{\S}{S}

\newcommand{\w}{\wedge}

\newcommand{\Dehn}{\tau}

\DeclareMathOperator{\im}{Im}
\DeclareMathOperator{\lk}{lk}
\DeclareMathOperator{\sympl}{sympl}
\DeclareMathOperator{\re}{Re}

\begin{document}
\title[Contact homology of left-handed stabilizations]{Contact homology of left-handed stabilizations and plumbing of open books}

\author{Frederic Bourgeois}
\address{(F.~Bourgeois) 
D\'epartement de Math\'ematiques, Universit\'e Libre de Bruxelles, CP 218, Boulevard du Triomphe, 1050 Bruxelles, Belgium
}
\email{fbourgeo@ulb.ac.be}
\author{Otto van Koert}
\address{(O.~van Koert)
Department of Mathematics, Hokkaido University, Kita 10 Nishi 8, 
Sapporo 060-0810, Japan
}
\email{okoert@math.sci.hokudai.ac.jp}


\begin{abstract}
We show that on any closed contact manifold of dimension greater than 1 a contact structure with vanishing contact homology can be constructed. The basic idea for the construction comes from Giroux. We use a special open book decomposition for spheres. The page is the cotangent bundle of a sphere and the monodromy is given by a left-handed Dehn twist. In the resulting contact manifold we exhibit a closed Reeb orbit that bounds a single finite energy plane. As a result, the unit element of the contact homology algebra is exact and so the contact homology vanishes. This result can be extended to other contact manifolds by using connected sums. The latter is related to the plumbing- or $2$-Murasugi sum of the contact open books. We shall give a possible description of this construction and some conjectures about the plumbing operation.
\end{abstract}

\maketitle

\section{Introduction}
A few years ago Giroux established a correspondence between open books and contact structures on manifolds. In dimension $3$ this correspondence showed that overtwisted contact manifolds need to have left-handed Dehn twists in the monodromy of a supporting open book. On the other hand, some constructions of open books involving left-handed Dehn twists give overtwisted manifolds. For example, a left-handed stabilization of any open book supporting a contact structure yields an overtwisted contact manifold. Inspired by this phenomenon Giroux proposed a way to generalize the notion of overtwistedness. Namely, one can construct open books that involve a left-handed Dehn-Seidel twist. This gives candidates for ``overtwisted'' manifolds in higher dimensions. 

On the other hand, in order to show that Giroux's ideas really give special contact manifolds in higher dimensions, one needs some kind of criterion.
 A topologically stable generalization of an overtwisted disk is hard to find in practice, but a very promising generalization of overtwistedness has been put forth by Niederkrüger \cite{Niederkrueger:PS}.
 Indeed, in a contact manifold $(M,\xi)$ of dimension $2n-1$ he considers contact structures that admit a family of overtwisted disks parametrized by a manifold $S$ of dimension $n-2$. Such contact manifolds are called {\bf $PS$-overtwisted}. Niederkrüger shows that $PS$-overtwisted contact manifolds are non-fillable.
 Moreover, by now $PS$-overtwisted contact structures are known to exist in every dimension greater than $1$ and there exist $PS$-overtwisted contact structures on spheres, see \cite{Presas:non-fillable} and \cite{Niederkrueger:non-fillable} for the construction of such manifolds.
 However, we shall be looking at other distinguishing qualities of overtwisted manifolds.

 A few criteria that come to mind are the facts that overtwisted contact manifolds are not symplectically fillable and that they have vanishing contact homology \cite{Yau:OT}.
 We shall consider the latter criterion, in particular because the vanishing of contact homology is conjectured to imply the non-existence of a strong filling.

At first, we shall consider the simplest case where one has an open book for $S^{2n-1}$ with page $T^*S^{n-1}$ and as monodromy a single left-handed Dehn-Seidel twist. The idea is to find a closed Reeb orbit that bounds a unique finite energy plane. As a result, the element $1$ in the contact homology algebra is exact and therefore all elements are exact. This result remains true under connected sums which illustrates the ``flexibility'' of this construction. 

This is related to the so-called plumbing or $2$-Murasugi sum, of which we shall now briefly describe the construction.
 Suppose we are given two contact open books with symplectic pages $\Sigma_1$ and $\Sigma_2$ with properly embedded Lagrangian balls $L_1\subset \Sigma_1$ and $L_2 \subset \Sigma_2$.
 Assume that the monodromy is given by $\psi_1$ and $\psi_2$ respectively.
 Suppose also that the boundaries $\partial L_i\subset \Sigma_i$ are Legendrian spheres. Then we can glue the pages by a plumbing construction. Simply put, we can find Weinstein neighborhoods of $L_i$ giving $q$ and $p$ coordinates. Then we can glue $\Sigma_1$ to $\Sigma_2$ by identifying the $q$-coordinates of $L_1$ with the $p$-coordinates of $L_2$ and vice versa. This yields a page $\Sigma$. For the monodromy of the new open book, we use $\psi_2 \circ \psi_1$.
 
This is a well known construction in dimension $3$ and a theorem of Torisu \cite{Torisu:Murasugi_sum} shows that the resulting manifold is contact and in fact contactomorphic to the connected sum $(\Sigma_1,\psi_1)\# (\Sigma_2,\psi_2)$.
 We shall give here some evidence that this might be true in higher dimensions as well.
 This would, besides being a useful tool for contact open books, also imply the following two claims.
 A right-handed stabilization of a contact open book is contactomorphic to the contact open book itself, see \cite{Giroux:ICM}.
 A left-handed stabilization of a contact open book has vanishing contact homology.

\subsection{Remarks and summary of our results}
We conclude the introduction by a few observations on vanishing of the contact homology and a summary of our main results.
 In this paper, we will be using the full contact homology, i.e.~the homology of the differential graded algebra generated by closed Reeb orbits as defined in~section~2.1 of \cite{Eliashberg:SFT}. Geometrically, we have the following well known observation for vanishing contact homology \cite{Eliashberg:ICM98}.
\begin{conjecture}
Suppose $(M,\xi=\ker \alpha)$ is a strongly fillable contact manifold. Then $HC_*(M,\xi)\neq 0$.
\end{conjecture}
The idea is here that a filling $(W,\omega)$ gives a cobordism of $(M,\alpha)$ to the ``empty contact manifold'', the latter having contact homology isomorphic to $\Q[H_2(W)]$. In a bit more detail, the filling gives a map 
\begin{eqnarray*}
\Phi: {\bf A}_*(M,\xi) & \longrightarrow & \Q[H_2(W)],
\end{eqnarray*}
defined by counting index $0$ holomorphic planes in the filling $(W,\omega)$.
 By the homomorphism property, this map sends $1\in {\bf A}_*(M,\xi)$ to $1\in \Q[H_2(W)]$.
 Furthermore, the map $\Phi$ is a chain map, i.e.~here it satisfies the identity
$$
\Phi \circ \partial_W=0.
$$
Here $\partial_W$ denotes the differential in the ``symplectization part'' of the filling. This is the differential of contact homology, but with coefficients in $\Q[H_2(W)]$ instead of $\Q[H_2(M)]$. The identity itself can be shown in the same way as the chain map property for cylindrical symplectic cobordisms needed for invariance of contact homology. If we suppose now that $HC_*(M,\xi)=0$, then we find a linear combination of products of Reeb orbits that have differential equal to $1$,
$$
\partial( \sum_{\mathrm{finite} } q_i e^{A_i}\gamma_{n^i_1}\ldots \gamma_{n^i_{k_i}} )=1.
$$
Let now $j$ denote the map on homology induced by the inclusion $M\subset W$. Then in terms of the $\partial_W$ differential, we have
$$
\partial_W( \sum_{\mathrm{finite} } q_i e^{j(A_i)}\gamma_{n^i_1}\ldots \gamma_{n^i_{k_i}} )=1.
$$
If we take $\Phi$ on both the left- and right hand side, we find $0$ on the one hand and $1$ on the other hand, which gives a contradiction. The above ideas give the argument to prove this conjecture up to transversality. Indeed, the chain map property of $\Phi$ has not yet been established rigorously.

The above conjecture gives a geometric application for vanishing of contact homology, namely the non-existence of a filling. 
We also would like to point out that, in dimension $3$, overtwisted contact manifolds have vanishing contact homology, see \cite{Yau:OT}. 
So far, these overtwisted manifolds are the only ones known to have vanishing contact homology. 
The following definition might therefore make sense.
\begin{definition}
We say a cooriented contact manifold $(M,\xi=\ker \alpha)$ is \emph{algebraically overtwisted} if $HC_*(M,\xi)=0$.
\end{definition}

In this language the main theorem that we shall prove is
\begin{theorem}
Let $(M,\xi)$ be a cooriented contact manifold and let $L$ be a boundary parallel Lagrangian ball in a compatible open book.
 Then the left-handed stabilization of $(M,\xi)$ along $L$ is algebraically overtwisted.
\end{theorem}
In fact, we have evidence to conjecture that algebraically overtwisted contact manifolds behave nicely under connected sums.
\begin{conjecture}
Let $(M,\xi)$ be a cooriented contact manifold and let $(N,\eta)$ be an algebraically overtwisted contact manifold. Then $(M,\xi)\#(N,\eta)$ is also algebraically overtwisted.
\end{conjecture}

\subsection{Plan of the paper}

The paper is organized as follows. 
 In Section~\ref{sec:prelims} we describe the setup for a special left-handed stabilization of $(S^{2n-1},\xi_0)$ in more detail.
 We shall describe an explicit model and give most of the needed geometric data.
 
In Section~\ref{sec_chain_complex} we describe the chain complex of contact homology in more detail and compute the degrees of involved Reeb orbits.
 Section~\ref{sec:hol_curves} is about holomorphic curves. By explicit computation we find a holomorphic curve, which is one of the ingredients of the differential. Section~\ref{sec:transversality} is the main technical part of the paper.
 There we shall show that the linearized Cauchy-Riemann operator at the solution we found earlier is surjective. The latter ensures a proper curve count.

 Section~\ref{sec:uniqueness} is concerned with the existence of other holomorphic curves.
 We shall show that the finite energy plane found in Section~\ref{sec:hol_curves} is the only curve that is asymptotic to a certain Reeb orbit of index $1$.
 This shows that the contact homology of our model manifold vanishes.

 Finally, in Section~\ref{sec:stabilization} we shall describe the stabilization procedure and Murasugi sum in more detail.
 We shall finish the paper by making a few conjectures about stabilizations in relation to contact homology and Murasugi sums in general.\\
\\
\noindent{\bf Acknowledgements. }
This work was initiated during a postdoc funded by the Fonds National de la Recherche Scientifique, Belgium. Currently, O.~van Koert is supported by the Japan Society for the Promotion of Science. 

We would like to thank Emmanuel Giroux for helpful comments and suggestions on plumbing of open books.

\section{Preliminaries and setup}
\label{sec:prelims}
We first describe a construction for an exotic contact sphere. 
The construction is the simplest case of an idea of Giroux, which is to modify the stabilization of an open book. 
By stabilization we mean the following. 
Let $P$ be the $2n$-dimensional page of the open book of $(M,\xi)$ and let $L$ be a Lagrangian $n$-disk in $P$ with $\partial L\subset \partial P$. 
Suppose that the monodromy of the open book is the identity on a neighborhood of $L$. 
We can attach a Weinstein $n$-handle to $P$ along $\partial L$ following \cite{Weinstein:surgery}. 
This way we obtain a Lagrangian sphere in the new page $\tilde P$. 
If we now choose as monodromy a right-handed Dehn twist on this Lagrangian sphere and compose this map with the original monodromy on the rest of $P\subset \tilde P$, then the resulting contact manifold is conjectured to be contactomorphic to $(M,\xi)$. 
Note that the latter is actually proved in many cases, for example it was shown by Giroux \cite{Giroux:ICM} in dimension~$3$ and it is also known for boundary parallel Lagrangian balls.
We provide an alternative description of a stabilization in Section~\ref{sec:stabilization}. 

However, Giroux's idea is to replace the right-handed Dehn twist in a stabilization by a left-handed one.
 For the diffeomorphism type of the resulting open book, this does not make any difference.
 For the contact structure there is a difference.
 In dimension 3 this is well known to give overtwisted contact manifolds, see for instance Lemma~4.1 in \cite{Goodman:OTopenbook}.
 We shall show that in the simplest case in higher dimensions we obtain algebraically overtwisted contact manifolds.
 In the following section we describe this simplest case, namely stabilization of the standard open book of $S^{2n-1}$ (having page $D^{n-1}$).
 Alternatively, one can also describe this particular left-handed stabilization by simply applying a left-handed Dehn twist to $T^*S^{n-1}$.
 
\subsection{Dehn twists}
As a initial model we consider $(T^*S^{n-1},d \lambda)$, where $\lambda$ is the canonical $1$-form on $T^*S^{n-1}$.
Throughout this paper we will be using coordinates
$$
(\vec q,\vec p) \in \R^{2n}
$$
to describe the cotangent bundle $T^*S^{n-1}$.
 These coordinates satisfy the relations
\begin{equation}
\label{eq:coordinatesTSn}
\vec q \cdot \vec q=1 \text{ and } \vec q \cdot \vec p=0.
\end{equation}
With these coordinates the canonical $1$-form $\lambda$ is given by
$$
\lambda=\vec p \cdot d \vec q.
$$
We use this symplectic manifold as page in an open book decomposition for $S^{2n-1}$. 
Next, we shall define a Dehn twist. First define the following auxiliary map describing the normalized geodesic flow
$$
\sigma_t(\vec q,\vec p)=
\left(
\begin{array}{cc}
\cos t & -|\vec p|^{-1} \sin t \\
|\vec p|\sin t & \cos t 
\end{array}
\right)
\left(
\begin{array}{c}
\vec q \\ \vec p
\end{array}
\right) .
$$
For $k\in \Z$ a $k$\textbf{-fold Dehn twist $\Dehn_k$} is a diffeomorphism of $T^*S^{n-1}$ of the form
$$
\tau_k(\vec q,\vec p)=\sigma_{g_k(|\vec p|)}(\vec q,\vec p)
$$

Here $g_k$ is a smooth function with the following properties.
\begin{itemize}
\item{} In 0, we have $g_k(0)=k\pi$.
\item{} Fix $p_0>0$. For $|\vec p|>0$, the function $g_k(|\vec p|)$ either decreases or increases to $0$ at $g_k(p_0)$, depending on whether $k$ is positive or negative.
For $|\vec p| \geq p_0$, we put $g_k(|\vec p|)=0$.
\end{itemize}
This is the standard definition of a Dehn-Seidel twist as introduced by Seidel, see Section~6 in \cite{Seidel:knotted_lagrangian}.

For computational convenience, it is useful to use another form instead. 
In fact, it will be convenient for us to compose the Dehn twist with a small geodesic flow increasing with distance.
This way, it will glue nicely to a standard model near the binding. 
We will now redefine $g_k$ to suit our purposes.
Fix a small constant $\tilde \epsilon>0$ and put $\tilde g_k:=g_k +\tilde \epsilon |\vec p|$.
In particular, for $k>0$, the function $\tilde g_k$ has no zeroes.
For $k<0$, the function $\tilde g_k$ has a unique zero at $p_0$.
By abuse of notation, we shall simply write $g_k$ to mean $\tilde g_k$ in the remainder of this paper.

We shall call the associated map $\tau_k$ a {\bf left-handed Dehn twist} if $k=-1$. 
For $k=1$, the map $\tau_k$ is called a {\bf right-handed Dehn twist}.

In Figure~\ref{fig_function_gk} we sketched a graph of a function satisfying these properties for our modification of a Dehn twist. 
In the following we will also use the quantity $g_k-k\pi$, which we will denote by $f_k$.

\begin{figure}
\begin{center}
\include{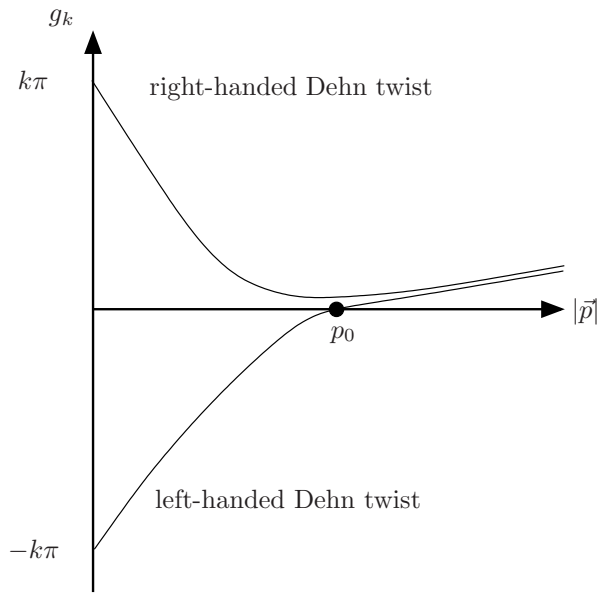}
\end{center}
\caption{Examples of the function $g_k$ describing a right- and left-handed Dehn twist}
\label{fig_function_gk}
\end{figure}

\begin{remark}
Our choice of $g_k$ ensures that the contact form we will be working with is Morse-Bott.
It is also helpful for the determining the differential of contact homology.
 See also Remark~\ref{rem:complicated_diff}.
\end{remark}

Note that a Dehn twist is a symplectomorphism, i.e.~$\tau_k^* d\lambda=d\lambda$. This follows from the transformation behavior of the canonical $1$-form $\lambda$ on $T^*S^{n-1}$. We have 
$$ 
\tau_k^* \lambda = \lambda + |\vec p|\,d\big(g_k(|\vec p|)\big). 
$$
Note that the difference $\lambda - \tau_k^* \lambda $ is exact, implying our above claim. As a primitive of this difference $\lambda - \tau_k^* \lambda $  we take 
$$
h_k(|\vec p|):=1+\int_0^{|\vec p|} s
g_k'(s)ds.
$$ 
Note that $h_k$ can be assumed to be positive on the interval $[0,1]$ by choosing $p_c<1$ sufficiently small.

\subsection{A simple open book decomposition for $S^{2n-1}$}
\label{sec_openbook_construction}
In this section we will apply Giroux's construction for contact open books with a $k$-fold Dehn twist, indicated by the subindex $k$. 
First we construct a mapping torus of $T^*S^{n-1}$ using a Dehn
twist following the construction of Giroux and Mohsen
\cite{Giroux:talk,Giroux:ICM}. 
Consider the map
\begin{eqnarray*}
\phi_k:~ T^*S^{n-1} \times \R & \longrightarrow & T^*S^{n-1}\times \R ,\\
(\vec q,\vec p;\varphi) & \longmapsto & (\tau_k(\vec q,\vec p);\varphi+h_k(|\vec p |)).
\end{eqnarray*}
This map preserves the contact form 
$$
\alpha=d\varphi+ \vec p d\vec q
$$ 
on $T^*S^{n-1}\times \R$, so we obtain an induced contact structure on  
$$
A:=T^*S^{n-1}\times \R / \phi_k.
$$
We define $A_0$ to be the subset of $A$ that corresponds to the zero section of $T^*S^{n-1}$,
$$
A_0:=\{ (\vec q,0;\varphi)\in A \} .
$$

We shall use another mapping torus in order to have an explicit model for gluing in the binding and to make computations more convenient. We shall consider
$$
\tilde A:=\left( \left( T^*S^{n-1}-0 \right) \times \R \right) /(x,\varphi)\sim (x,\varphi+1) \cong \left( T^*S^n-0 \right) \times S^1.
$$
Away from the zero section a Dehn twist can be written as the flow of a vector field (namely the geodesic flow). Its $t$-flow is given by $\sigma_t$. We use this fact to "unwrap" the Dehn twist giving the following map

\begin{eqnarray*}
\psi: \tilde A & \to & A \\
\left( (\vec q,\vec p);\varphi \right) & \mapsto & \left( \sigma_{\varphi g_k }(\vec q,\vec p);h_k(|\vec p|) \varphi \right) 
\end{eqnarray*}

Note that this map is a diffeomorphism onto its image. We get a contact form $\tilde \alpha$ on $\tilde A$ by pulling back $\alpha$ under the map $\psi$,
$$
\tilde \alpha= \psi^* \alpha=\tilde h_k(|\vec p|) d\varphi+\vec p d \vec q,
$$
where $\tilde h_k=1-\int_0^{|\vec p|} g_k (s) ds$.
We will denote a neighborhood of the binding of the open book decomposition of $S^{2n-1}$ by $B:=ST^*S^{n-1}\times D^2$. We can, in fact, choose this neighborhood so large that it covers all of $\tilde A$. In other words $B$ will describe the entire contact manifold save for a set of positive codimension. The computations we shall do require explicit coordinates for $ST^*S^{n-1}$. More precisely, we have defined the set $ST^*S^{n-1}\subset \R^{2n}$ in the above by the equations
\begin{equation}
\label{eq:STS_definition}
\vec q \cdot \vec q=1,~\vec p \cdot \vec p=1 \text{ and } \vec q \cdot \vec p=0.
\end{equation}

We can choose this neighborhood so large that it includes the entire mapping torus of pages, except for the zero section, as we shall see in the following.
On this neighborhood $B$ of the binding we have the contact form
\begin{equation}
\label{eq:contactform_binding}
\alpha=h_1(r) \lambda+h_2(r)d\varphi,
\end{equation}
where $\lambda$ is the restriction of the canonical $1$-form $\vec p d\vec q$ to $ST^*S^{n-1}$ and $(r,\varphi)$ are polar coordinates on $D^2$. The functions $h_1$ and $h_2$ are chosen in such a way that $\alpha$ is a contact form and such that it matches the form $\tilde \alpha$ in a collar neighborhood of the boundary. We have used the map
\begin{eqnarray*}
\psi_{b}:~B &\longrightarrow & \tilde A
\end{eqnarray*}
\begin{eqnarray}
\label{eq:binding_pullback}
(\vec q,\vec p,r,\varphi) & \longmapsto & (\vec q,\vec p/r,\varphi)
\end{eqnarray}
to identify these collar neighborhoods of the boundary.
 In Figure~\ref{fig_functions_left} we indicate how a graph of the functions $h_1$ and $h_2$ could look like such that the above holds. 
 The point $r_0$ where $h_2$ assumes its maximum is indicated in Figure~\ref{fig_functions_left}.
 For $\alpha$ to be a contact form, we need to impose that following condition on the functions $h_1$ and $h_2$,
$$
\alpha\w d\alpha^{n-1}=h_1^{n-2}(h_1h_2'-h_2 h_1')/r (dr \w r d\varphi \w \lambda \w d\lambda^{n-2}) \neq 0.
$$
That is to say that $\alpha$ is a contact form if $h_1\neq 0$ and the quantity
$$
\det H:=h_1h_2'-h_2h_1'
$$
is such that the smooth function $\det H/r$ is non-vanishing. This is the determinant of the matrix
\begin{equation}
\label{eq:definition_H}
H:=
\left(
\begin{array}{cc}
h_2' & h_1'\\
h_2 & h_1
\end{array}
\right).
\end{equation}
For the choice indicated in Figure~\ref{fig_functions_left}, we have indeed that $\det H/r\neq 0$.
\begin{figure}
\begin{center}
\include{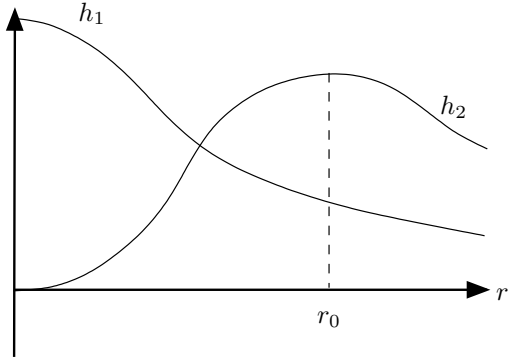}
\end{center}
\caption{Functions $h_1$ and $h_2$ for a left-handed Dehn twist pushed into the binding}
\label{fig_functions_left}
\end{figure}

\begin{remark}
Our choice of the function $h_2$ is not standard.
 It is more common to restrict the domain of $h_2$ to a set $[0,r_c]$, where $r_c<r_0$.
Roughly speaking, our choice of $h_2$ thickens the binding to include a larger part of the mapping torus $A$. 
If we want, we can in fact map the entire mapping torus $A$, except for the zero set of $T^*S^{n-1}$, into $B$.
The main advantage is that the holomorphic curves that we shall write down, can be written down in the set $B$, so there will be no need to match pieces of holomorphic curves lying in $A$ and $B$.
The role of the mapping torus $\tilde A$ is only an auxiliary one.
\end{remark}

This way we get a global contact form $\alpha$ on $S^{2n-1}$ that is compatible with a left-handed Dehn twist.
 We shall write the associated contact structure on $S^{2n-1}$ as $\xi=\ker \alpha$.
 We shall denote the contact sphere constructed this way by $(S^{2n-1},\alpha)$.

\begin{remark}
We would like to emphasize that the contact forms that we have constructed are all of Morse-Bott type. 
In the next section we will investigate the spaces that consist of closed Reeb orbits.
Finally, in the remainder of the paper we shall keep the notation $g_k$ and all other functions with a subscript $k$, but we shall always mean that $k=-1$, though some claims hold true in more general cases.
\end{remark}

\section{Chain complex}
\label{sec_chain_complex}
The goal is to do some index computations to show that for a left-handed Dehn twist there exists a degree $1$ orbit. For reasons of clarity and generality we keep the notation $g_k$. 

\subsection{Closed Reeb orbits and their Maslov indices}
As is usual in Morse-Bott contact homology we start by looking at the spaces of closed Reeb orbits. We shall show that they are closed submanifolds and we shall also compute their homology, which is needed for the computation of the degree of the generators of the Morse-Bott complex.

On the mapping torus $A$ the Reeb field for the contact form $\alpha$ is equal to
$$
R_{\alpha}=\frac{\partial}{\partial \varphi}.
$$
Note that the flow of this vector field preserves the norm of $\vec p$. 
It does \emph{not} preserve the coordinates $(\vec q, \vec p)$ because of the monodromy we used in defining $A$. 
Now observe that each turn around the binding advances the geodesic flow by $g_k(|\vec p|)$. 
Since a Reeb orbit in $A$ is closed if and only if the geodesic flow is a multiple of $2\pi$ after some turns around the binding, we see that we get closed orbits precisely when $g_k(|\vec p|) \in \pi \Q$. 
The period of a closed Reeb orbit with $g_k(|\vec p|) \in \pi \Q$ is $h_k(|\vec p|) m$, where $m$ is the smallest positive integer such that $m g_k(|\vec p|) \in 2\pi \Z$.

For us only the case where $g_k(|\vec p|)=0$ shall turn out to be important, which can only happen if $k<0$. Then we have the orbit space $ST^*S^{n-1}$ since the monodromy is the identity for those values of $\vec p$ with $g_k(|\vec p|)=0$. If $n$ is odd, its rational homology is given by
$$
H_j(ST^*S^{n-1};\Q)\cong \left\{ \begin{array}{cc}
0 & \text{ if }j\neq 0,2n-3 \\
\Q & \text{ if }j=0,2n-3. \\
\end{array}\right. 
$$
If $n$ is even, then we have
$$
H_j(ST^*S^{n-1};\Q)\cong \left\{ \begin{array}{cl}
0 & \text{ if }j\neq 0,n-2,n-1,2n-3 \\
\Q & \text{ otherwise.} \\
\end{array}\right.
$$
It turns out that we shall only need the lowest degree orbit (since this will kill contact homology). In other words, we can choose a Morse function on the orbit space with a unique minimum. The closed Reeb orbit that corresponds to that minimum has degree $1$ as we shall compute in the next section.

\subsubsection{Maslov index of closed Reeb orbits}
For the computations we shall use the mapping torus $\tilde A$. First of all, let us define a few auxiliary vector fields.
The time $t$-flow of the following vector field generates the geodesic flow $\sigma_t(\vec q,\vec p)$ for $|\vec p|>0$
$$
G:=|\vec p| \vec q \frac{\partial}{\partial \vec p}-\frac{1}{|\vec p|}\vec p \frac{\partial}{\partial \vec q}.
$$
Define the "radial" vector field
$$
P:=\frac{\vec p}{|\vec p|} \frac{\partial}{\partial \vec p}
$$
and the modified geodesic flow,
$$
Q:=-G-\frac{|\vec p|}{\tilde h_k}\frac{\partial}{\partial \varphi}.
$$
Note that both the $P$ and $Q$ vector fields lie in the contact structure $\xi$.

The Reeb field of the form $\tilde \alpha$ on $\tilde A$ is given by
$$
R=\frac{1}{\tilde h_k(|\vec p|)-|\vec p| \tilde h_k'(|\vec p|)}(\frac{\partial}{\partial \varphi}+\tilde h_k'(|\vec p|) G).
$$
In order to simplify the coming equation, we introduce two functions $N$ and $g$ such that the Reeb field looks like
$$
R=N\frac{\partial}{\partial \varphi}+g G.
$$
The time $t$-flow is then given by
\begin{eqnarray*}
Fl_t^R:~(\varphi,\vec q,\vec p) & \mapsto & (\varphi+Nt,\sigma_{gt}(\vec q,\vec p)).
\end{eqnarray*}

Now let $\gamma$ be a closed Reeb orbit. 
Note that the norm of the $\vec p$ coordinate is constant along Reeb orbits and that for a closed Reeb orbit, the coordinate $\vec p$ needs to satisfy $g(|\vec p|)/N(|\vec p|)\in \pi \Q$. 
To fix some notation for the closed Reeb orbit, we will consider the $T=i/N(|\vec p|)$-flow of the Reeb field such that orbit closes, where $i$ is an integer divisible by $m$.
Note that $i$ indicates how many times an orbit revolves around the binding. 
To compute the Maslov index, we choose a symplectic basis for $\xi$ along this closed Reeb orbit $\gamma$ consisting of
$$
P,Q,\vec r_l \frac{\partial}{\partial \vec p},\vec r_{l'} \frac{\partial}{\partial \vec q} \text{ for }l,l'=1,\ldots n-1,
$$
where the vectors $\vec r_l$ are chosen orthogonal to $\vec q$ and $\vec p$ and to $\vec r_{l'}$ for $l\neq l'$. This can be arranged using the Gram-Schmidt process. Note that this symplectic basis does not extend to a disk. Later on, we shall choose a different, $\varphi$-dependent combination of $P$ and $Q$ such that the resulting basis does extend over a disk.

With respect to the above basis of $\xi$, we can compute that the linearized flow looks like
$$
\psi(t)=
\left(
\begin{array}{cccc}
1 & 0 & 0 & 0 \\
-g't & 1 & 0 & 0 \\
0 & 0 & \cos(gt)\mathbbm{1} & |\vec p| \sin(gt)\mathbbm{1} \\
0 & 0 & -\sin(gt)/|\vec p| \mathbbm{1} & \cos(gt) \mathbbm{1}
\end{array}
\right)
=
\left(
\begin{array}{cccc}
1 & 0 & 0 & 0 \\
-g't & 1 & 0 & 0 \\
0 & 0 & \mathbbm{1} & 0\\
0 & 0 & 0 & \mathbbm{1}
\end{array}
\right)
.
$$
Since the vector $P$ and $Q$ do not extend to a disk in a natural way (they are analogous to $\frac{\partial}{\partial r}$ and $\frac{\partial}{\partial \varphi}$ in polar coordinates), we consider a different linear combination of $P$ and $Q$,
$$
\left(
\begin{array}{c}
P' \\
Q'
\end{array}
\right)
=
\Phi(\varphi) \left(
\begin{array}{c}
P \\
Q
\end{array}
\right)
=
\left(
\begin{array}{cc}
\cos(2\pi \varphi) & -\sin(2\pi \varphi) \\
\sin(2\pi \varphi) & \cos(2\pi \varphi)
\end{array}
\right)
\left(
\begin{array}{c}
P \\
Q
\end{array}
\right).
$$
The resulting basis of $\xi$ given by
$$
P',Q',\vec r_l \frac{\partial}{\partial \vec p},\vec r_{l'} \frac{\partial}{\partial \vec q} \text{ for }l,l'=1,\ldots n-1
$$
does extend over a disk. 
We can now easily compute the Maslov index of the closed Reeb orbit using standard formulas. 
We shall denote the Robbin-Salamon index by $\mu$ and the Maslov index for a loop by $\mu_l$, see \cite{Robbin:Maslovindex} and \cite{Salamon:Floerlecture_notes}.
The map $\Phi$ winds $i$ times around $0$, and the map $\psi$ consists of a single symplectic skew.  
Hence we see by the loop axiom (see \cite{Salamon:Floerlecture_notes}) that
$$
\mu(\gamma_0)=\mu(\psi)+\mu_l(\Phi)=1/2~{\rm {sign}}(g')+2i.
$$

\subsubsection{Grading}
We now compute the gradings using the standard perturbation procedure from \cite{Bourgeois:thesis},~Section~2.2. In other words we choose a Morse function on the orbit spaces to define a regular contact form. As indicated in Figure~\ref{fig_function_gk}, the function $g_k'$ is positive at that point, so the lowest possible grading of the orbits $\gamma$ of that orbit type is given by
$$
\deg(\gamma)=2i+1/2-1/2(2n-3)+(n-3)=2i-1.
$$

In particular, we see that in the Morse-Bott case a degree $1$ orbit appears.
 In the next section, we shall show that this orbit bounds a finite energy plane.

\section{Holomorphic curves}
\label{sec:hol_curves}

\subsection{A complex structure for $\xi$ and finite energy planes bounding closed Reeb orbits}
The main goal here is to describe finite energy planes.
 To that end, we need to define a suitable almost complex structure on the symplectization $\sympl(S^{2n-1},\alpha)$ of $(S^{2n-1},\alpha)$.
 First we define a complex structure on $\xi$, which we will then extend to an almost complex structure on $\sympl(S^{2n-1},\alpha)$ in the usual way.

We shall construct a finite energy plane converging to a closed Reeb orbit $\gamma_0$ with $|\vec p|=p_0$. As in the previous section $p_0$ is the unique zero of the function $g_k$. For the left-handed Dehn twist $k=-1$, but we will keep the notation with a $k$, because some arguments are more general.

On $B$, the thickened neighborhood of the binding, we take a complex structure for $\xi$ of the form
$$
J_\xi=
\left(
\begin{array}{cccc}
0 & 0 & 0 & \frac{h_1}{\det H} \\
0 & \vec q \vec p^T & \mathbbm{1} & -\frac{h_2}{\det H}\vec p \\
0 & -\mathbbm{1} & -\vec p \vec q^T & \frac{h_2}{\det H}\vec q \\
-h_2' & -h_1'\vec p^T & 0 & 0
\end{array}
\right).
$$
We have ordered the coordinates $(\varphi,\vec q,\vec p,r)$.
 This complex structure was inspired by the complex structure on $T^*S^{n-1}$, but adapted to our situation.
 Note that the above $J_\xi$ is only a complex structure for $\xi$ when restricted to submanifolds we consider, that is~$ST^*S^{n-1}\times D^2$;
 it is \emph{not} a complex structure on $\R^{2n+2}$.
 In other words, one needs to use Equation~\ref{eq:STS_definition}, its differential and of course the restriction to vectors in $\xi$ to verify that $J_\xi$ is a compatible complex structure for $d\alpha$.

To understand what is happening, it is useful to keep the following vector fields in mind. We have the Reeb field on the binding which generates the geodesic flow.
$$
R_\lambda:=\vec p \frac{\partial}{\partial \vec q}- \vec q \frac{\partial}{\partial \vec p}.
$$
The Reeb field of the contact form $\alpha$ on $B$ is
\begin{equation}
\label{eq:Reeb_binding}
R_\alpha=\frac{1}{\det H}\left(h_2' R_\lambda-h_1'\frac{\partial}{\partial \varphi} \right).
\end{equation}
Another useful vector field is
\begin{equation}
\label{eq:geofield_binding}
J \frac{\partial}{\partial r}=\frac{1}{\det H}\left(-h_2 R_\lambda+h_1\frac{\partial}{\partial \varphi} \right).
\end{equation}

We extend $J_\xi$ to an almost complex structure on the symplectization using the usual recipe by requiring
$$
J\frac{\partial}{\partial t}=R_\alpha.
$$
Here $t$ is the $\R$-coordinate on the symplectization and $R_\alpha$ is the Reeb field for the contact form $\alpha$. The extended complex structure looks like
\begin{equation}
\label{eq:J_binding}
J=
\left(
\begin{array}{ccccc}
0 & 0 & 0 & \frac{h_1}{\det H}  & -\frac{h_1'}{\det H}\\
0 & \vec q \vec p^T & \mathbbm{1} & -\frac{h_2}{\det H}\vec p & \frac{h_2'}{\det H} \vec p \\
0 & -\mathbbm{1} & -\vec p \vec q^T & \frac{h_2}{\det H}\vec q & -\frac{h_2'}{\det H}\vec q \\
-h_2' & -h_1'\vec p^T & 0 & 0 & 0 \\
-h_2 & -h_1\vec p^T & 0 & 0 & 0 
\end{array}
\right) .
\end{equation}
Here we have ordered the coordinates as $(\varphi,\vec q,\vec p,r,t)$.

Of course, we have not defined an almost complex structure near the zero-section of $T^*S^{n-1}$, but we can choose any extension, since the holomorphic curves that we shall consider, stay away from the zero-section. This is shown in Section~\ref{sec_holcurve_in_sphere_uniqueness}.

\subsubsection{Finite energy planes intersecting the binding}
Let us now parametrize a candidate finite energy plane $u$ in polar coordinates, say $(\rho,\psi)$, and make the ansatz that the curve looks like
$$
u(\rho,\psi)=(\psi,\vec q(\rho),\vec p(\rho),r(\rho),t(\rho))
$$
where we have used the same ordering of the coordinates as in Formula \eqref{eq:J_binding}.
With this ansatz, the Cauchy-Riemann equation reads
$$
u_\rho=-J\frac{1}{\rho}\frac{\partial}{\partial \psi}=\frac{1}{\rho}(h_2'\frac{\partial}{\partial r}+h_2\frac{\partial}{\partial r}).
$$
We immediately see that both $\vec q$ and $\vec p$ are constant along the holomorphic curve. The $r$ and $t$ coordinate satisfy the ordinary differential equations
\begin{equation}
\label{eq:CR_r_coor}
r_\rho=\frac{1}{\rho}h_2'(r)
\end{equation}
\begin{equation}
\label{eq:CR_t_coor}
t_\rho=\frac{1}{\rho}h_2(r).
\end{equation}
One can solve this system by first integrating the first equation and substituting the solution in the second. There are two integration constants. The one for the equations for $r$ corresponds to reparametrizations in the domain. The integration constant for the $t$ component represents the translation symmetry of the solution in the $\R$ direction of the symplectization.

Note that, although the holomorphic plane depends on the choice of $h_1$ and $h_2$, the projection to the contact manifold does not. Let us denote the holomorphic plane we obtained this way by $u_0$.

\begin{remark}
\label{rem:complicated_diff}
If we assume uniqueness of the holomorphic plane and transversality at $u_0$ (which we shall show later) we see that $\partial \gamma_0=1$. This is sufficient for vanishing contact homology, but, of course, there are different ways for this to happen as the following example shows. Choose the functions $h_1$ and $h_2$ as indicated in Figure~\ref{fig_functions_left2}. The resulting contact form is isotopic to the one we used earlier. We still find a finite energy plane with $r$ coordinate smaller than $r_0$, but in addition we find a holomorphic cylinder going from the closed Reeb orbit $\gamma_0$ at $r_0$ to a closed Reeb orbit $\gamma_1$ at $r_1$. Hence we get $\partial \gamma_0=1+\gamma_1$. But we also find a cylinder from the closed Reeb orbit $\gamma_2$ at $r_2$ to the $\gamma_1$. In fact, $\partial \gamma_2=\gamma_1$. So we still have vanishing contact homology as expected. All these holomorphic curves can be found in the same way as we found the finite energy plane, namely we can use the above ansatz.

Now we can revisit Figure~\ref{fig_function_gk} and observe that we could have chosen a decreasing slope for the function $g_k$ in our definition of Dehn-twist. This would have resulted in the existence of a degree $0$ orbit and a degree $1$ orbit for the standard sphere, which we get for right-handed Dehn twist, $k=+1$. In that case we would have found not only a finite energy plane bounding the degree $1$ orbit, but also a holomorphic cylinder between the degree $1$ and the degree $0$ orbit. The details are similar to the above example. 
\begin{figure}
\begin{center}
\include{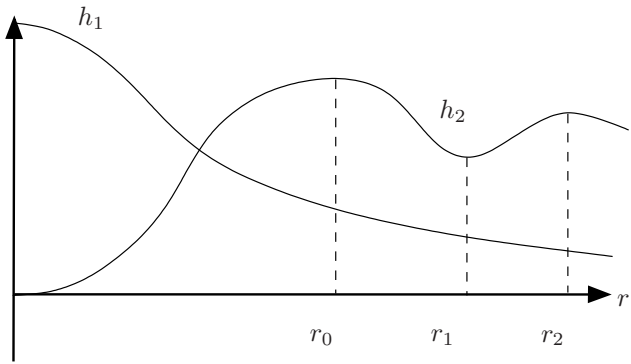}
\end{center}
\caption{Functions $h_1$ and $h_2$ for a more complicated differential}
\label{fig_functions_left2}
\end{figure}

\end{remark}

\section{Transversality for finite energy planes}
\label{sec:transversality}
In order to ensure that we get a proper curve count, we need to establish surjectivity of the linearized Cauchy-Riemann operator at $u_0$. We do this by computing the kernel of the linearized equation. We shall show that the dimension of the kernel coincides with the Fredholm index, and therefore the linearized Cauchy-Riemann operator is surjective at $u_0$.

First of all, we observe that the solution $u_0$ has constant $\vec q$ and $\vec p$ components. This is helpful when doing computations later on. Let $\Xi=(\xi_\varphi,\vec \xi_q,\vec \xi_p,\xi_r,\xi_t)$ be a section of $u_0^*T \sympl(M)$. Note that $ST^*S^{n-1}\times D^2$ can be considered as a submanifold of $\R^{2n+2}$. This allows us to do all computations in $\R^{2n+2}$. The (extended) linearized Cauchy-Riemann operator looks like
\begin{equation}
\label{eq:linCR_operator}
D\Xi=\partial_\rho \Xi+\Xi(J)_{u_0}\frac{1}{\rho}\partial_\psi u_0+J(u_0)\frac{1}{\rho}\partial_\psi \Xi.
\end{equation}
This expression only makes sense in $\R^{2n+2}$. In order to make statements about the original problem, we need to impose the condition that $\Xi$ be contained in $u_0^*T \sympl(M)$. 
Note that $\partial_\psi u_0=\frac{\partial}{\partial \varphi}$ and that both the $\vec q$ and $\vec p$ coordinates are constant along the solution $u_0$. These properties are of use to simplify the linearized equations. The asymptotic boundary conditions are most easily given in suitable cylindrical coordinates around the puncture. We describe how to do this in the following interlude.

\subsection{Asymptotic boundary conditions for the Cauchy-Riemann operator}
\label{sec:boundary_conditions}
The notions here are taken from \cite{Bourgeois:Coherent_orientations} and \cite{Dragnev:Fredholm}. 
First we shall describe the situation for a generic contact form, i.e.~a contact form for which all Reeb orbits are non-degenerate. 
We shall use Sobolev spaces with exponential weights. 
For this, we always choose cylindrical coordinates $(\rho,\psi)\in [R,\infty ) \times S^1$. 
Then we define 
$$
w\in W^{1,p}_\delta([R,\infty ) \times S^1,\R^N) \text{ if and only if } e^{\delta \rho} w \in W^{1,p}([R,\infty ) \times S^1,\R^N).
$$ 
The positive number $\delta$ has to be chosen smaller than the spectral gap of the linearized operator. 
See \cite{Dragnev:Fredholm} for an alternative description to obtain the weight factor $\delta$. 
\subsubsection{Setup for the non-degenerate case}
In this section we shall briefly describe the Banach manifold setup. 
Since we are looking at the special case of finite energy planes, we can choose global cylindrical coordinates near the puncture.
We shall take the following cylindrical coordinates $(\rho,\psi)$ for $\C$,
\begin{eqnarray*}
\R \times S^1 & \longrightarrow \C 
\end{eqnarray*}
\begin{eqnarray}
\label{eq:cylinder_coordinates}
(\rho,\psi) & \longmapsto & e^\rho e^{i\psi}
\end{eqnarray}
Note that in general the situation is more complicated than here.
Here, we will be considering the Banach manifold $\mathcal B^{1,p}_\delta(\gamma_0)$, the space of maps $u:\C \to \R \times S^{2n-1}$ such that
\begin{itemize}
\item{} $u$ is locally in $W^{1,p}$.
\item{} Written in the cylindrical coordinates of Equation~\eqref{eq:cylinder_coordinates}, the components of the map $u=(a;\vartheta,\zeta)\in \R\times \S^{2n-1}$ satisfy
$$
(a-T\rho-a_0),(\vartheta-\psi-\vartheta_0),\zeta \in W^{1,p}_\delta
$$
for some $a_0,\vartheta_0$. The coordinates $\vartheta$ and $\zeta$ are coordinates for a tubular neighborhood $S^1\times D^{2n-2}$ of the Reeb orbit $\gamma_0$ with $\vartheta$ representing the $S^1$ coordinate. The number $T$ is the action of the Reeb orbit $\gamma_0$. 
\end{itemize}

For the linearized equation, it is important to discuss the tangent space $T_u \mathcal B^{1,p}_\delta(\gamma_0)$. The latter can be identified with the vector space
$$
T_u\mathcal B^{1,p}_\delta(\gamma_0)\cong \Gamma_\delta^{1,p}(\C,u^*T(\R\times \S^{2n-1})),
$$
which is the space of sections of $u^*T(\R\times \S^{2n-1})$ that satisfy
\begin{itemize}
\item{} Each section is in $W^{1,p}_\delta(\C,u^*T(\R\times S^{2n-1}))\oplus \R \frac{\partial}{\partial t} \oplus \R R_\alpha$. 
\end{itemize}

\subsubsection{Setup for the Morse-Bott case}
In the Morse-Bott case, the setup is as follows. 
We have the Banach manifold $\mathcal B^{1,p}_{MB,\delta}(S_0)$ which is the space of maps $u:\C \to \R \times S^{2n-1}$ such that
\begin{itemize}
\item{} $u$ is locally in $W^{1,p}$.
\item{} There is a Reeb orbit $\gamma \in S_0$ and real numbers $a_0,\vartheta_0$ such that the following holds. 
In the cylindrical coordinates of Equation~\eqref{eq:cylinder_coordinates} the components of the map $u=(a;\vartheta,\zeta)\in \R\times \S^{2n-1}$ satisfy
$$
(a-T\rho-a_0),(\vartheta-\psi-\vartheta_0),\zeta \in W^{1,p}_\delta.
$$ 
The coordinates $\vartheta$ and $\zeta$ are coordinates for a tubular neighborhood $S^1\times D^{2n-2}$ of the Reeb orbit $\gamma$ with $\vartheta$ representing the $S^1$ coordinate. 
The number $T$ is the action of the Reeb orbit $\gamma$. 
\end{itemize}
For the Morse-Bott case, the tangent space $T_{u}\mathcal B^{1,p}_{MB,\delta}(S_0)$ can be identified with the vector space of sections of $u^*T(\R\times \S^{2n-1})$ that satisfy
\begin{itemize}
\item{} Each section $\Xi$ is in $W^{1,p}_\delta(\C,u_{\gamma}^*T(\R\times S^{2n-1}))\oplus \R \frac{\partial}{\partial t} \oplus \R R_\alpha \oplus T_\gamma S_0$.
The last factor, $T_{\gamma}S_0$, can be interpreted as follows. Take a vector $v \in T_{\gamma}S_0$. 
Then we can lift this vector $v$ to a section $\tilde v$ of $TS^{2n-1}$ along $\gamma$.
By taking suitable lifts (for instance as in Chapter 5 of \cite{Bourgeois:thesis} of a basis of $T_{\gamma}S_0$ the above statement makes sense;
a section $\Xi$ can be written as a linear combination of $\partial_t$, $R_\alpha$, these lifts and a section in $W^{1,p}_\delta(\C,u_{\gamma}^*T(\R\times S^{2n-1}))$.
\end{itemize}

\subsubsection{Relations between the non-degenerate case and the Morse-Bott case}
\label{sec:Morse-Bott_relation}
We can get from the Morse-Bott case to the non-degenerate case by making a small perturbation. 
If the perturbation is small enough, then there is a relation between holomorphic curves in both cases by the implicit function theorem.

In particular, for the dimensions of the kernels of the linearized equations we have the following relation
$$
\dim \ker D_{MB}=\dim \ker D_{non-deg} +\dim S_0.
$$
Here $D_{MB}$ denotes the linearized operator in the Morse-Bott case and $D_{non-deg}$ denotes the linearized operator in a nearby perturbed non-degenerate case. 
$S_0$ is the orbit space for the Morse-Bott closed orbits.

\subsubsection{Admissibility of solutions}

Since we shall directly work with the linearized differential equation, it is useful to make the following definition, which we will only apply in the non-degenerate case.
\begin{definition}
We say a solution $\Xi$ to the (linear) equation
$$
\bar \partial_J \Xi=0
$$
is {\bf admissible} if $\Xi \in W^{1,p}_\delta(\C,u^*T(\R\times S^{2n-1}))\oplus \R \frac{\partial}{\partial t} \oplus \R R_\alpha$.
\end{definition}
In other words admissible solutions are precisely those that satisfy the linearized Cauchy-Riemann equation and lie in $T_u \mathcal B^{1,p}_\delta(\gamma_0)$.

\begin{remark}
A necessary requirement for an admissible solution $\Xi$ is that there are constants $C,D$ such that
$$
(\Xi-C\frac{\partial}{\partial t}-D R_\alpha)
$$
has exponential decay. This criterion is easier to check since we can ignore the derivative. Furthermore, it turns out that this criterion is sufficient for us.
\end{remark}

In the next sections we shall be computing the kernel of operator \eqref{eq:linCR_operator}. We shall do these computations as if we had a non-degenerate contact forms. In particular, translations along the Morse-Bott orbit space are not admissible solutions.

\subsection{Kernel of the linearized Cauchy-Riemann operator}
Next, we want to compute the kernel of the operator \eqref{eq:linCR_operator}. We shall use components 
$$\Xi=(\xi_\varphi,\vec \xi_q,\vec \xi_p,\xi_r,\xi_t)$$ for Equation~\eqref{eq:linCR_operator}. Keep in mind that we have extended the Cauchy-Riemann equations to $\R^{2n+2}$, so we need to impose the conditions
$$
\vec q \cdot \vec \xi_q=0,~\vec p \cdot \vec \xi_p=0,~\vec p \cdot \vec \xi_q+ \vec q \cdot \vec \xi_p=0
$$
to ensure that a solution $\Xi$ is tangent to $T \sympl(M)$.
 This can be seen by taking the differential of Equation~\eqref{eq:STS_definition} and plugging in $\Xi$.
 Now let us look at components of $\Xi$ that are normal to the solution $u_0$. Note that the vectors
\begin{equation}
\label{eq:normal_vectors}
(0,\vec r,0,0,0) \text{ and } (0,0,\vec r,0,0)
\end{equation}
are always normal to $u_0$ and lie in the tangent space $T \sympl(M)$ if $\vec r$ is both orthogonal to $\vec q$ and $\vec p$. Here it is important to note that $\vec q$ and $\vec p$ are constant along the solution $u_0$. If we take the standard inner product on $\R^{2n+2}$ of Equation~\eqref{eq:normal_vectors} and Equation~\eqref{eq:linCR_operator}, we find the equations
$$
\vec r \cdot \partial_\rho \vec \xi_q+\frac{1}{\rho} \vec r \cdot \partial_\psi \vec \xi_p=0
$$
and
$$
\vec r \cdot \partial_\rho \vec \xi_p-\frac{1}{\rho} \vec r \cdot \partial_\psi \vec \xi_q=0.
$$
These are both standard Cauchy-Riemann equations, so their solutions are holomorphic functions. 
In a Morse-Bott setup, we actually get solutions here. Indeed, constant solutions correspond to moving the asymptotics of the holomorphic curve along the Morse-Bott orbit space. 
However, in the non-degenerate case, these solutions are \emph{not} admissible, since they have neither exponential decay nor do they correspond to either translation invariance or rotation along the Reeb orbit near the puncture. Hence we conclude that 
$$
\vec r \cdot \vec \xi_q=\vec r \cdot \vec \xi_p=0.
$$
As a result the only components which could have a non-zero solution are those along the directions $\partial_r,J\partial_r,\partial_t,R_\alpha$. This leaves the equations for $\xi_\varphi$, $\xi_{\parallel}:=\vec p \cdot \vec \xi_q=-\vec q \cdot \vec \xi_p$, $\xi_r$ and $\xi_t$, which we assemble in the system
\begin{equation}
\label{eq:system_lin_CR}
\left\{ 
\begin{array}{c}

\partial_\rho \xi_\varphi+\frac{1}{\rho}\frac{h_1}{\det H}\partial_\psi \xi_r-\frac{1}{\rho}\frac{h_1'}{\det H}\partial_\psi \xi_t=0,\\

\partial_\rho \xi_{\parallel}-\frac{1}{\rho}\frac{h_2}{\det H}\partial_\psi \xi_r+\frac{1}{\rho}\frac{h_2'}{\det H} \partial_\psi \xi_t=0,\\

\partial_\rho \xi_r-\frac{1}{\rho}h_2' \partial_\psi \xi_\varphi-\frac{1}{\rho}h_1' \partial_\psi \xi_{\parallel}=\frac{h_2''}{\rho}\xi_r,\\

\partial_\rho \xi_t-\frac{1}{\rho}h_2\partial_\psi \xi_\varphi-\frac{1}{\rho}h_1 \partial_\psi \xi_{\parallel}=\frac{h_2'}{\rho} \xi_r.
\end{array}
\right.
\end{equation}
\begin{remark}
Before we rewrite this system, observe that constant $\xi_\varphi$, $\xi_t$ and $\xi_{\parallel}$ with $\xi_r=0$ give solutions to the above system. None of these have exponential decay, and only constant $\xi_\varphi$ (rotation along a Reeb orbit) and constant $\xi_t$ (translation invariance) are admissible. Constant $\xi_{\parallel}$ does not fix the asymptotics, but corresponds to the moving along the orbit space and is hence not admissible in our sense. Note that together with non-admissible constant solutions we found previously, these vectors span the tangent space to the orbit space, which illustrates Section~\ref{sec:Morse-Bott_relation}. 
\end{remark}

With the following notation and Equation~\eqref{eq:definition_H} we simplify the above equations
$$
Y=\left( \begin{array}{c} \xi_\varphi \\ \xi_{\parallel} \end{array} \right)
\text{, and }
Z=\left( \begin{array}{c} \xi_r \\ \xi_t \end{array} \right)
,
$$
then we can rewrite the differential equation as
$$
\partial_\rho Y + \frac{1}{\rho} \partial_\psi H^{-1} Z=0
$$
$$
\partial_\rho Z -\frac{1}{\rho} \partial_\psi H Y=\frac{\xi_r}{\rho}\left( \begin{array}{c} h_2'' \\ h_2' \end{array} \right) .
$$
From now on, we shall indicate components of a vector by a subindex. If we define $\tilde Z=H^{-1} Z$, we can rewrite the latter equation (using $\partial_\rho r =h_2'/\rho$) as
$$
\partial_\rho \tilde Z -\frac{1}{\rho} \partial_\psi Y=\frac{\tilde Z_2(h_1'h_2''-h_2'h_1'')}{\det H \rho} \left( \begin{array}{c} h_1 \\ -h_2 \end{array} \right) .
$$
\begin{remark}
Here, one should observe that $H^{-1}$ is actually only defined for $r>0$ (or equivalently $\rho>0$). Hence the new system is not quite the same as the old one, but they differ only at $\rho=0$. We shall find general solutions of the new system and transform them back to see whether they are smooth at $\rho=0$. 
\end{remark}
Now put $W:=\tilde Z+iY$. This allows us to put all the remaining equations into a nice form,
\begin{equation}
\label{eq:linCR_simple}
\partial_\rho W + i \frac{1}{\rho} \partial_\psi W=\frac{\tilde Z_2(h_1'h_2''-h_2'h_1'')}{\det H \rho} \left( \begin{array}{c} h_1 \\ -h_2 \end{array} \right).
\end{equation}

\subsection{Automorphisms and symmetries}
We know that infinitesimal automorphisms of plane give rise to solutions of Equation~\eqref{eq:linCR_operator}. In addition, the translation symmetry of the symplectization also gives rise to a solution of~\eqref{eq:linCR_operator}. We assemble these solutions in the $5$-dimensional vector space $S$. It is generated by $\C^2$ valued functions that correspond to these symmetries,
$$
S:=\langle
(1,0),(i,0),(1/z,0),(i/z,0),(-h_1'/\det H,h_2'/\det H)
\rangle.
$$

\begin{lemma}
Elements of vector space $S$ are solutions to~\eqref{eq:linCR_simple} that are admissible in the sense of Section~\ref{sec:boundary_conditions}. Also, any solution $W$ to~\eqref{eq:linCR_simple} can be decomposed as
$$
W=A+B,
$$
where $A$ satisfies the exponential decay condition and $B\in S$.
\end{lemma}

In polar coordinates $(\rho,\psi)$ infinitesimal automorphisms of the plane are given by
$$
(\rho,0) \text{ and } (0,1)
$$
which correspond to automorphisms of the form $z\mapsto az$, and by
$$
(\cos \psi,-\frac{\sin \psi}{\rho}) \text{ and } (\sin\psi,\frac{\cos \psi}{\rho}),
$$
which correspond to automorphisms of the form $z\mapsto z+b$
We can map these infinitesimal automorphisms to solutions of Equation~\eqref{eq:linCR_operator} by applying $Tu_0$. In terms of the transformed system~\eqref{eq:linCR_simple}, we get solutions
$$
W=(1,0) \text{ and } W=(i,0)
$$
corresponding to $z\mapsto az$ and solutions
$$
W=(\cos\psi/\rho-i\sin\psi/\rho,0)=(1/z,0) \text{ and }W=(i\cos\psi/\rho+\sin\psi/\rho,0)=(i/z,0)
$$
corresponding to $z\mapsto z+b$.
In addition, the translation symmetry of the symplectization also gives rise to a solution
$$
W=(-h_1'/\det H,h_2'/\det H)
$$
to the Equation~\eqref{eq:linCR_simple}. 

In the next section we shall show that elements in $\ker D_{u_0}$ correspond to elements in $S$. 
This implies surjectivity of $D_{u_0}$, because then the kernel of $D_{u_0}$ has the same dimension as the index. 
Indeed, we can apply Theorem $9$,~formula~(23), from Dragnev~\cite{Dragnev:Fredholm} to see that the index is $5$. Here we can use that $\mu(\gamma_0)+n-3=1$ following our computations in Section~\ref{sec_chain_complex}. 

Alternatively, we know the virtual dimension of the moduli space of finite energy planes bounding the Reeb orbit $\gamma_0$ is given by $1$. Since the automorphism group of the plane is $4$-dimensional, the index of the corresponding Fredholm problem is $5$.

\subsection{No other solutions}
Here we adapt an argument due to Salamon and Zehnder~\cite{SalamonZehnder:Morse}, Proposition~4.2. The result is slightly different from that of Salamon and Zehnder: we show that under suitable assumptions there exist only certain $\psi$-dependent solutions to the linearized Cauchy-Riemann equation. This will restrict the behavior of any solution to the linearized Cauchy-Riemann equation considerably. To state the result, we need to define the following averages. Let $W$ be a function from $\R \times S^1\to \R^{2n}$. Then we define
$$
W_k(\rho,\psi):=e^{ik\psi} \int_{0}^{2\pi} W(\rho,\psi') e^{-ik\psi'} d\psi'.
$$
In other words, we simply pick out the $k$-th Fourier component of $W$. By looking at Parseval's identity, we immediately get the following lemma.
\begin{lemma}
\label{lemma:SZlike_inequality}
Let $W$ be a function from $\R \times S^1\to \R^{2n}$ and define
$$
\bar W=W-W_0-W_1-W_{-1}.
$$
Then 
$$
2 \| \bar W \|\leq \| \partial_\psi \bar W \|.
$$
\end{lemma}

Before we continue, we shall first introduce a weight function that it is adapted to our functional-analytic setup. Let us first choose two constants $\rho_0<\rho_\infty$. Then define a smooth function $w:\R \to \R$ with the following properties:
\begin{itemize}
\item{} For $\rho<\rho_0$, we put $w(\rho)=2$.
\item{} For $\rho>\rho_\infty$, we put $w(\rho)=\delta$.
\end{itemize}
We use the measure $e^{w(\rho)} d\rho d\psi $ as weight on $\R \times S^1$. 
Note that this gives a weight on $\C$ via the map given in Formula~\eqref{eq:cylinder_coordinates}.
The induced weight is standard near $0\in \C$, but far away from $0$, this weight corresponds to an asymptotic weight $e^{\delta \rho}$ in cylindrical coordinates.

The next lemma provides the main argument to show that solutions to Equation~\eqref{eq:linCR_simple} must have a special form. Consider the operator $F$
\begin{eqnarray*}
F:~W^{1,2}_\delta (\C;\R^{2n}) & \longrightarrow & L^{2}_\delta(\C;\R^{2n}) \\
W & \longmapsto & \partial_\rho W+J_0 \partial_\psi W-A W.
\end{eqnarray*}
Here we use cylindrical coordinates $(\rho,\psi)$ for $\C$ from Equation~\eqref{eq:cylinder_coordinates}. 
As for the assumptions of the lemma, note that by modifying the functions $h_1$ and $h_2$ that appear in the definition of the contact form, we can always arrange that $\| A \| <C$ for any constant $C>0$.
Indeed, we see from Equation~\ref{eq:linCR_simple} that scaling $h_1$ and $h_2$ changes the norm of $A$.

\begin{lemma}
\label{lemma:psi-dependence}
Let $W$ be an element in $\ker F$ with with $W_0=0$ and $W_{\pm 1}=0$. Suppose that $A$ only depends on $\rho$ and $\| A \| <2$. Then $W=0$.
\end{lemma}
\begin{proof}
The proof consists mostly of the argument due to Salamon and Zehnder, but we need to take the exponential weights into account. We shall use the measure of the form $e^{w(\rho)}d\rho d\psi$ that we just described. Note that we can assume the derivative $\max \partial_\rho w$ to be arbitrarily small by choosing a suitable function $w$.

The goal is to show that the function $W$ does not depend on $\psi$, which implies that it vanishes.
\begin{align*}
\| \partial_\psi W \|^2 &\leq \| \nabla W \|^2\\
~& =\int \sum_i \langle \partial_i W,\partial_i W \rangle e^{w(\rho)}d\rho d\psi\\
~& =-\int\sum_i\left( \langle W,\partial_i^2 W \rangle + \langle W,\partial_i w \partial_i W \rangle \right) e^{w(\rho)}d\rho d\psi \\
~& =-\int\left( \langle W,\partial_J \bar \partial_J W \rangle + \langle W,\partial_\rho w \partial_\rho W \rangle \right) e^{w(\rho)}d\rho d\psi \\
~& =\int\left( \langle \bar \partial_J W,\bar \partial_J W \rangle +\langle \partial_\rho w W,\bar \partial_J W \rangle- \langle W,\partial_\rho w \partial_\rho W \rangle \right) e^{w(\rho)}d\rho d\psi \\
~& =\int \left( |\bar \partial_J W|^2+ \langle W,\partial_\rho w J_0 \partial_\psi W \rangle \right) e^{w(\rho)}d\rho d\psi \\
~& \leq \| \bar \partial_J W \|^2+\| W \| \| \partial_\psi W \| \max \partial_\rho w \\
~& \leq \| D_{u_0} W \|^2 + 2\| D_{u_0} W\| \| A W\|  + \| A W\|^2 + \| \partial_\psi W \|^2 \max \partial_\rho w /2 \\
~& \leq \left( \| A \|^2 /4+\max \partial_\rho w /2 \right ) \| \partial_\psi W \|^2.
\end{align*}
In the last two steps we have used Lemma~\ref{lemma:SZlike_inequality}.
Finally, the factor $\| A \|^2 /4+\max \partial_\rho w /2$ is smaller than $1$ by assumption. Hence $\| \partial_\psi W \|=0$, which implies the claim given the assumptions.
\end{proof}

This lemma implies that the only solutions to Equation~\eqref{eq:linCR_simple} are those that are $\psi$-independent and those that have a Fourier series with terms of the form $\cos \psi$, $\sin \psi$. Higher frequency terms can never be solutions. Therefore we see that the kernel of the linearized operator~\eqref{eq:linCR_operator} is at most $12$-dimensional. We shall now show that the actual kernel is only $5$-dimensional, i.e.~it only consists of elements in $S$.

\begin{lemma}
\label{lemma:solutionslinCR}
Any solution $W$ to Equation~\eqref{eq:linCR_simple} that satisfies the boundary conditions lies in $S$.
\end{lemma}
\begin{proof}
The main idea is to use the Fourier decomposition for a solution, which will transform the partial differential equation into an ordinary differential equation. This ordinary differential equation will then show that solutions that do not lie in $S$ and that are smooth near $\rho=0$ must explode for large $\rho$. This is incompatible with the boundary conditions.

Let $W$ be a solution to Equation~\eqref{eq:linCR_simple}. We can simplify the computations by assuming that $h_1=1-r^2$ and $h_2=r^2$ for small $r$. This means that $W$ is holomorphic near $\rho=0$. If we forget about boundary conditions then  we see that the equation has a $12$-dimensional solution space by using Lemma~\ref{lemma:psi-dependence}. These have the form
$$
W=\left(
\begin{array}{c}
a_0+a_1 z+a_{-1}z^{-1}\\
0
\end{array}
\right)
\text{ and near }\rho=0 \text{, }
W=\left(
\begin{array}{c}
0 \\
b_0+b_1 z+b_{-1}z^{-1}
\end{array}
\right) .
$$ 
The solutions $S$ to Equation~\eqref{eq:linCR_simple} correspond to linear combinations of the above form with $a_{1}=b_{-1}=b_{1}=0$ and $b_0$ real. Note that in order to for any solution to satisfy the boundary conditions we need $a_1=0$. Also note that $b_{-1}=0$ for a smooth solution. We shall now argue that, for solutions satisfying the boundary conditions, $b_0$ is real and $b_1=0$. This will imply that a solution lies in $S$.

We see that $b_0$ is real, because the function
$$
W=\left(
\begin{array}{c}
0 \\
i
\end{array}
\right)
$$
does not satisfy the boundary conditions, but it does satisfy Equation~\eqref{eq:linCR_simple} everywhere. Let us now consider the Fourier transform of Equation~\eqref{eq:linCR_simple}. We shall only need the part corresponding to $W_2$. If we write $W_2=b_0+c(\rho) \cos \psi +d(\rho) \sin(\psi)$, we see
$$
\partial_\rho c+i/\rho d=-h_2\frac{h_1'h_2''-h_2'h_1''}{\rho \det H} \re c
$$
$$
\partial_\rho d-i/\rho c=-h_2\frac{h_1'h_2''-h_2'h_1''}{\rho \det H} \re d.
$$
From the above we know that $\re c= \im d$ and $\im c=-\re d$ near $\rho=0$. Define 
$$
H_2=-h_2(h_1'h_2''-h_2'h_1'')/{\det H}.
$$ 
Let us consider the equation for $(\re c,\im d)$, which can be written as
$$
\left(
\begin{array}{c}
\re c\\
\im d
\end{array}
\right) '=
\left(
\begin{array}{cc}
H_2/\rho & 1/\rho \\
1/\rho & 0
\end{array}
\right)
\left(
\begin{array}{c}
\re c\\
\im d
\end{array}
\right).
$$
Now look at the phase plane of this system for varying $\rho$, see Figure~\ref{fig:phaseplane}.
\begin{figure}
\begin{center}
\include{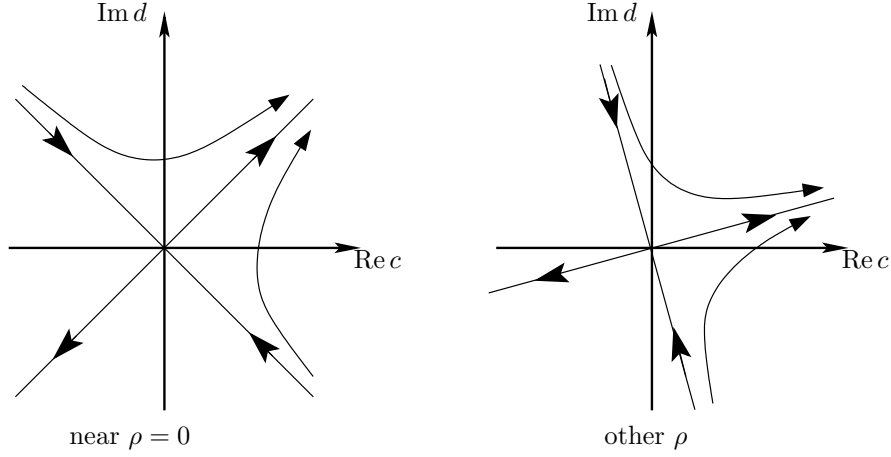}
\end{center}
\caption{Phase plane for $\re c$ and $\im d$; the cross indicates the eigenvectors for contracting and expanding directions. Note that all flow lines point into the first quadrant}
\label{fig:phaseplane}
\end{figure}
Indeed, we see that any solution that starts in the first quadrant, i.e.~$\re c>0$ and $\im d>0$, stays there. We also see that the $\re c$ necessarily goes to infinity for large $\rho$. This can be deduced from the eigenvectors of the matrix
$$
\left(
\begin{array}{cc}
H_2 & 1 \\
1 & 0
\end{array}
\right)
,
$$
which are given by
$$
\left(
\begin{array}{c}
1  \\
-H_2/2 +\sqrt{H_2^2/4+1}  
\end{array}
\right)
\text{ for expanding eigenvalue }
H_2/2 +\sqrt{H_2^2/4+1},
$$
and
$$
\left(
\begin{array}{c}
1  \\
-H_2/2 -\sqrt{H_2^2/4+1}  
\end{array}
\right)
\text{ for contracting eigenvalue }
H_2/2 -\sqrt{H_2^2/4+1}.
$$
This behavior of the eigenvalues shows that all flow lines point into the first quadrant for all values of $\rho$. In particular, it follows that $\re c=0$ and hence $\im d=0$. For the other pair of Fourier coefficients, $\im c$ and $\re d$ a similar argument holds.
\end{proof}

\begin{remark}
Lemma~\ref{lemma:psi-dependence} is in fact not really necessary here, since the proof of Lemma~\ref{lemma:solutionslinCR} can be extended to all Fourier coefficients. This is a bit clumsier and the method of Lemma~\ref{lemma:psi-dependence} might have a much wider range of applications.
\end{remark}

\subsection{Transversality for other curves}
The above computation shows that we have transversality at the curve that we found explicitly. In the next section we shall show that there are no other finite energy planes nor any curves with $\gamma_0$ at the positive puncture. This means that $\partial \gamma_0=1$ and also $\partial^2 \gamma_0=0$.

To get $\partial^2=0$ for all products of Reeb orbits, we need to choose some perturbation scheme. In general, we can not simply perturb $J$ to obtain a differential. We note here that there are several perturbation schemes in preparation \cite{Hofer:Polyfolds,HWZ:PolyfoldsI,HWZ:PolyfoldsII} and \cite{CM:GW_Transversality} that should provide the required identity $\partial^2=0$. Because we have shown transversality for curves involving $\partial \gamma_0$, the choice of any perturbation scheme does not affect $\partial \gamma_0=1$ if the perturbation can be chosen small enough.

\section{Contact homology of left-handed Dehn twists}
\label{sec:uniqueness}
In the following two sections, we shall show that, under suitable circumstances, there are no other planes bounding the orbit $\gamma_0$. 
In order to see that there are no other holomorphic curves with $\gamma_0$ at the positive puncture, we simply use an energy argument. Since $\gamma_0$ is the only closed orbit that has linking number $1$ with the binding, we just adjust the functions $h_1$, $h_2$ and possibly $g_k$ to ensure that $\gamma_0$ is the orbit with lowest action. That shows that $\gamma_0$ cannot bound holomorphic curves other than planes. Modulo the uniqueness result of the planes, we have that
$$
\partial \gamma_0=\pm 1.
$$
This implies that the contact homology algebra vanishes. Indeed, if $\gamma$ represents an element in contact homology, then 
$$
\partial (\gamma_0 \gamma)=(\partial \gamma_0)\gamma+(-1)^{\deg \gamma} \gamma_0( \partial \gamma)=\pm \gamma.
$$
Hence every element $\gamma$ that represents a cycle is already a boundary, so the contact homology algebra is trivial.

\subsection{Uniqueness of finite energy planes bounding the orbit $\gamma_0$ in the sphere}
\label{sec_holcurve_in_sphere_uniqueness}
First we show that any finite energy plane is contained in the region 
$$
U:=\{ (x,r,\varphi)\in ST^*S^{n-1}\times D^2 ~|~ r\leq r_0\} ,
$$
where $r_0>0$ is the point with $h_2 '(r_0)=0$. We shall use polar coordinates $(r,\varphi)$ for the disk $D^2$. 
In dimension $3$ the argument can be considerably simplified, because then the contact form is a closed form when restricted to the set 
$$
\partial U = \{ (x,r,\varphi)\in ST^*S^{1}\times D^2 ~|~ r=r_0\},
$$
which allows the use of more homotopical arguments.

In general, the idea is simply that we can obtain an energy estimate by using some topological data, in this case the winding number around the binding, which is always $1$.

Now let $u$ be a finite energy plane. Now denote the preimage under $u$ by
$$
V:=u^{-1}(U).
$$
The $r$-component of $u$, which we shall denote by $u_r$, is a smooth function $V\subset \C \to \R$. After possibly reparametrizing $u$, we can ensure that $0$ maps to the unique intersection point of $u$ with the binding $ST^*S^{n-1}\times \{ 0 \}$. 

This intersection point is unique, because the symplectization of the binding is an almost complex manifold. Hence we have that the algebraic intersection number of $u$ with $\R \times ST^*S^{n-1}\times \{ 0 \}$ is larger or equal to the geometric one by positivity of intersection. Since we also know that the linking number of the Reeb orbit $\gamma_0$ with $ST^*S^{n-1}\times \{ 0 \}$ is equal to $1$, the algebraic and geometric intersection numbers are both equal to $1$.

For a \emph{regular} value $r'<r_0 $ of $u_r$, the preimage $u_r^{-1}(r')$ is a collection of circles. 
Take $\gamma_{r'}$ to be such a circle at which $u_r=r'$ that, in addition, bounds a disk containing $0$. 
In the subsequent arguments we shall always parametrize these circles with constant $u_r$ value by $\tilde \psi\in [0,2\pi)$ such that the circle has winding number $+1$ around $0$.  
Such a parametrization is not unique, but we fix one by using
\begin{eqnarray*}
S^1 & \longrightarrow & ST^*S^{n-1}\times D^2 
\end{eqnarray*}
\begin{eqnarray}
\label{eq:circle_param}
\tilde \psi & \longmapsto & (f_r(\tilde \psi);r,\tilde \psi) 
\end{eqnarray}
Note that this parametrization is consistent with the linking number of any such curve with the binding (which is the above winding number). 
We get the tangent vector $\partial_{\tilde \psi}$ to the circle. Note that $\tilde \psi$ is not necessarily related to the coordinate $\psi$ we used in previous sections. It does play a similar role though. Also consider the vector $\nu:= -i \partial_{\tilde \psi}$, which plays the role of outward pointing normal similar to $\partial_\rho$. 

We write a tangent vector in the image of the circle $\gamma_r$ as
\begin{equation}
\label{eq:r_circle_decomposition}
\partial_{\tilde \psi} u=c_r(\tilde \psi)J \frac{\partial}{\partial r}+d_r(\tilde \psi)R_\alpha+e_r(\tilde \psi)\frac{\partial}{\partial t}+X_\lambda,
\end{equation}
with $X_\lambda \in \ker \lambda$, where $\lambda$ is the contact form on the binding. 
See Formula~\eqref{eq:contactform_binding}. 
Furthermore, we have used Formulas~\eqref{eq:Reeb_binding} and \eqref{eq:geofield_binding}.
We shall use this decomposition later on as well. Also notice that $\nu(u_r)=c_r(\tilde \psi)$, since
$$
\nu u+J \frac{\partial u}{\partial \tilde \psi}=0.
$$
Note that we can find such a decomposition as in Equation~\eqref{eq:circle_param} on an annulus in $U$ with regular values of $u_r$. 
We shall do this in Section~\ref{sec:annulus_energy}.

In terms of the Parametrization~\eqref{eq:circle_param} we get the following useful relations for a circle in the preimage of a regular value of $u_r$.
The winding number along such a curve $\gamma_r$ is given by
\begin{equation}
\label{eq:windingnumber_relation}
\frac{1}{2\pi}\int_{\gamma_r} d\varphi= 
\frac{1}{2\pi}\int_{\tilde \psi=0} ^{2\pi}\frac{1}{\det H} (h_1(r) c_r(\tilde \psi)-h_1'(r)d_r(\tilde \psi))d\tilde \psi=1.
\end{equation}
This can be shown by using Equations~\eqref{eq:Reeb_binding} and \eqref{eq:geofield_binding}.
The formula gives a relation between $c_r$ and $d_r$ that we shall exploit in the next section. 
The action of a circle $\gamma_r$ is given by
\begin{equation}
\label{eq:action_special_parametrization}
\mathcal A(\gamma_r)=\int_0^{2\pi} d_r(\tilde \psi) d\tilde \psi.
\end{equation}

\subsubsection{Energy of annuli}
\label{sec:annulus_energy}
Let us now estimate the energy of a subset $C$ of a holomorphic curve in $U$ by first omitting the non-negative term $h_1 d\lambda$,
$$
\int_{C}d\alpha\geq \int_{C} h_1' dr\w \lambda+\int_{C} h_2'dr\w d\varphi =E_1(C)+E_2(C).
$$
Note that if $C$ is an annulus with inner radius $r_1$ and outer radius $r_2$ the integral
$$
E_2(A_{r_1,r_2}) =\int_{A_{r_1,r_2}} h_2'dr\w d\varphi
$$
is always positive. For the first integral $E_1$ this is not true. 
Now parametrize all regular values of $u_r$. This is a union of annuli in $U$ which we shall denote by $U_A$. 
We parametrize those annuli in the set $u(V)$ which is a set of full measure,
\begin{eqnarray*}
 U_A & \longrightarrow & ST^*S^{n-1}\times D^2\\
(r,\tilde \psi) & \mapsto & (f_r(\tilde \psi),r,\tilde \psi).
\end{eqnarray*}

Now consider a single annulus of regular values in $U_A$ with inner radius $r_1$ and outer radius $r_2$.
We can estimate the first integral for this annulus as
\begin{equation}
\label{eq:energy_annulus}
E_1(A_{r_1,r_2})=\int_{r=r_1}^{r_2} \int_{\tilde \psi=0} ^{2\pi} h_1'(r)d\lambda(\partial_{\tilde \psi}u) d\psi dr=
\end{equation}
$$
\int_{r=r_1}^{r_2} \int_{\tilde \psi=0} ^{2\pi}h_1'(r) \frac{1}{\det H}\left( h_2'd_r(\tilde \psi)-h_2 c_r(\tilde \psi) \right) d\tilde \psi dr.
$$
To see the last step, simply pick out the $R_\lambda$ part of $\partial_{\tilde \psi}u$ by using Equations~\eqref{eq:Reeb_binding}, \eqref{eq:geofield_binding} and \eqref{eq:r_circle_decomposition}. Two shorthand notations are convenient,
$$
\bar c_r=\int_{\tilde \psi=0} ^{2\pi}c_r(\tilde \psi) d\tilde \psi,~~\bar d_r=\int_{\tilde \psi=0} ^{2\pi}d_r(\tilde \psi) d\tilde \psi.
$$
By performing the $\tilde \psi$-integral in Equation~\eqref{eq:energy_annulus} and using the relation for the winding number from Equation~\eqref{eq:windingnumber_relation}
$$
\frac{1}{\det H} \left( h_1(r) \bar c_r-h_1'(r) \bar d_r \right) =2\pi,
$$
we can simplify Equation~\eqref{eq:energy_annulus} to get the following relation for a part of the energy of an annulus piece of the holomorphic curve,
$$
E_1(A_{r_1,r_2})= \int_{r=r_1}^{r_2} \left( \bar c_r -2\pi h_2'(r) \right) dr=\int_{r=r_1}^{r_2}\frac{-h_1'}{h_1}(2\pi h_2-\bar d_r) dr.
$$
Note that this gives a non-negative contribution if $\bar d_r\leq 2\pi h_2(r)$. This allows us to complete the argument. 

Take an increasing sequence $r_i$ of regular values of $u_r$ converging to $r_0$ and consider the disk $A_{0,r_i}$, which is a subset of the holomorphic curve $u$ with $r$-coordinate less than $r_i$.
There are two cases.
\begin{itemize}
\item{} There is a sequence $\{ r_i \}$ that satisfies $d_{r_i}\geq 2\pi h_2(r_i)$.
The above estimate cannot be used, but we see directly
$$
E(A_{0,r_i})=d_{r_i}\geq 2\pi h_2(r_i)
$$
by using Stokes' theorem and Equation~\eqref{eq:action_special_parametrization}. For $i\to \infty$, we get
$$
E(A_{0,r_0})\geq 2\pi h_2(r_0).
$$
\item{} If the above case does not hold true, then we can assume that the sequence $\{ r_i \}$ satisfies $d_{r_i}\leq 2\pi h_2(r_i)$.
Indeed, if the sequence has infinitely many $r_i$ where this is not true, then we are again in the above case.
Hence we can assume that for $u_r>\tilde r$ all regular values $r'$ of $u_r$ satisfy $d_{r'}\leq 2\pi h_2(r')$.
By the above estimate and by using Sard's theorem, i.e.~regular values have full measure, we can compute the energy of the annulus $A_{r_1,r_0}$.
Indeed, we can sum the contributions from the above estimates to obtain
$$
E(A_{r_1,r_0})\geq E_2(A_{r_1,r_0})=2\pi \left( h_2(r_0)-h_2(r_1) \right) .
$$
On the other hand, we can assume that
$$
E(A_{0,r_1})\geq 2\pi h_2(r_1)
$$
by decreasing $r_1$ until either $r_1=0$ or until $d_{r_1}\geq 2\pi h_2(r_1)$. 
We may pass through non-regular values, because this has no influence on the energy.
\end{itemize}
Therefore we always have
$$
\int_{u(V)}d\alpha\geq 2\pi h_2(r_0)=\mathcal A(\gamma_0).
$$
If the holomorphic curve would have points with $r$-coordinate larger than $r_0$, then there would also be open sets giving a positive contribution to the $d\alpha$-energy, increasing the energy beyond $\mathcal A(\gamma_0)$, which is impossible. Hence we see that the holomorphic curve must satisfy $u_r\leq r_0$.

\noindent{\bf Reduction to dimension 3}\\
The above energy computations show that a finite energy plane only moves in the directions $\partial_t$, $\partial_r$, $\partial_\varphi$ and the direction of the geodesic flow, $\vec p \partial_{\vec q}-\vec q \partial_{\vec p}$, since otherwise
$$
\int h_1 d\lambda>0,
$$
which would increase the energy above $\mathcal A(\gamma_0)$. In dimensions $3$ we can use coordinates $(t,\theta,r,\varphi)$ around the symplectization of the binding
$$
\R\times S^1\times D^2.
$$
The coordinate $\theta$ corresponds to the geodesic flow. This means that a finite energy plane bounding $\gamma_0$ gives also rise to a finite energy plane in the symplectization of a $3$-dimensional sphere. We can simply map the $t$, $r$ and $\varphi$ coordinates of a finite energy plane in a higher dimension contact manifold to the corresponding coordinates in the $3$-dimensional manifold. The amount of geodesic flow can be mapped to the $\theta$ coordinate.

Hence the existence of a finite energy plane different from $u_0$ in dimension $3$ is equivalent to the existence of a finite energy plane different from $u_0$ in higher dimensions.

In dimension $3$ we can use an intersection theoretic argument to show that there are no other planes bounding $\gamma_0$. Let $u$ be any finite energy plane bounding $\gamma_0$. It has the form
\begin{eqnarray}
\label{eq:u_general_plane}
u:~(\rho,\psi) & \longmapsto & (f_t(\rho,\psi),f_\theta(\rho,\psi),f_r(\rho,\psi),f_\varphi(\rho,\psi)).
\end{eqnarray}
Note that $u_0$ has the form
\begin{eqnarray*}
u_0:~(\rho,\psi) & \longmapsto & (H_1(\rho),\theta_0,H_2(\rho),\psi)
\end{eqnarray*}
for a constant $\theta_0$ and $H_1$ and $H_2$ solving the Equations~\eqref{eq:CR_r_coor} and~\eqref{eq:CR_t_coor} . We shall now show that we can assume that $f_\theta$ is not constant.

\noindent{\bf Constant $f_\theta$}\\
We show that if the $f_\theta$-component of $u$ is equal to $\theta_0$, then $u$ is equivalent to $u_0$. Take a regular value $r_0$ of the $r$-component of $u$. We take a circle in the preimage of $r_0$, which we shall parametrize by $\tilde \psi$. We can again use the decomposition from~\eqref{eq:r_circle_decomposition}. Note that, because we are now in dimension~$3$, it follows that $X_\lambda=0$. Now we use that the $\theta$-coordinate is constant along $u$; to see the $\theta$-component of a vector we use the Equations~\eqref{eq:Reeb_binding} and~\eqref{eq:geofield_binding}. Then, looking at Equation~\eqref{eq:r_circle_decomposition} we see that $e_r(\tilde \psi)=0$, for otherwise $J\partial_{\tilde \psi} u$ would have a non-zero $\theta$-component. This means that $\partial_{\tilde \psi}u$ has only components in the $\varphi$ direction.  

As a result, the plane $u$ goes through circles at $r=r_0$ with constant $t$ and $\theta$ coordinate. Hence we can translate $u$ in the $t$-direction such that $u$ intersects $u_0$ along at least a circle at $r=r_0$. However, then we get a contradiction to positivity of intersection if we assume $u$ and $u_0$ to be not equivalent. Indeed, both $u$ and $u_0$ are simple, so according to Proposition~E.2.2 in \cite{McDuff:J-hol_symplectic_topology} the intersection points of $u$ and $u_0$ are isolated. Since they are not by the above, it follows that $u$ is equivalent to $u_0$.

\noindent{\bf Non-constant $f_\theta$}\\
Now we consider the case that $f_\theta$ is not constant, so in particular $u$ is not equivalent to $u_0$. We shall show that this leads to a contradiction.
By assumption, both $u$ and $u_0$ are asymptotic to $\gamma_0$. On the other hand, if $u$ is a solution to the Cauchy-Riemann equations of the form~\eqref{eq:u_general_plane}, then for $c\neq 0$
\begin{eqnarray*}
u_c:~(\rho,\psi) & \longmapsto & (f_t(\rho,\psi),f_\theta(\rho,\psi)+c,f_r(\rho,\psi),f_\varphi(\rho,\psi))
\end{eqnarray*}
is also a solution to the Cauchy-Riemann equations, but one that is asymptotic to another orbit, say $\gamma_c$. This Reeb orbit $\gamma_c$ is not linked with $\gamma_0$, because it does not intersect the Seifert surface determined by $u_0$. So $\lk(\gamma_c,\gamma_0)=0$.

On the other hand, the linking number can also be computed as a $4$-dimensional intersection number of the Seifert surfaces of $\gamma_c$ and $\gamma_0$. Indeed, $\lk(\gamma_c,\gamma_0)=u_c \cdot u_0$. But if $f_\theta$ is non-constant, we can find a small $c$ such that $u_c$ and $u_0$ intersect. By positivity of intersection, it follows that $\lk(\gamma_c,\gamma_0)>0$, which gives a contradiction. Hence we conclude that $u_0$ is the only finite energy plane bounding $\gamma_0$.

This completes the argument that there is a unique finite energy plane bounding $\gamma_0$. As as result, the contact homology algebra $HC_*(S^{2n-1},\xi_L=\ker\alpha)$ vanishes.

\begin{remark}
Alternatively, we can use a counting argument to show that there are no other finite energy planes in the region $r\leq r_0$ contributing the differential. Indeed, the almost complex structure is $\varphi$-independent, i.e.~the complex structure $J$ has an $S^1$-symmetry.
 Therefore, if $u$ is any finite energy plane bounding $\gamma_0$, then we can rotate $u$ in the $\varphi$ direction.
 This fixes the boundary condition, i.e.~the Reeb orbit $\gamma_0$ is invariant under the $\varphi$-rotation. If $u$ is different from $u_0$, then this symmetry gives us an $S^1$-family of holomorphic curves. Such an $S^1$-family is not counted by the differential, see~\cite{Bourgeois:thesis} for more details. Note however that the previous argument shows that such an $S^1$-family cannot even exist.
\end{remark}

\subsection{Connected sums with an exotic sphere}
Let $(M_1,\alpha_1)$ and $(M_2,\alpha_2)$ be cooriented contact manifolds of dimension $2n-1$. It is well known that the connected sum $M_1 \# M_2$ is also a contact manifold \cite{Weinstein:surgery}. One can choose Darboux balls in $M_1$ and $M_2$ and connect them via a connecting tube which has a very explicit model.

This allows us to retain some control on the Reeb dynamics of $M_1\# M_2$. Orbits that do not pass the connected sum region are not affected. Furthermore, there are new orbits that lie entirely in the connecting tube, which we shall call tube orbits, and wandering orbits, i.e.~orbits that start out in $M_1$, go to $M_2$ via the tube and back again. These wandering orbits can have any degree, but their action can be made arbitrarily large. 

The tube orbits have been studied by Ustilovsky in his thesis \cite{Ustilovsky:thesis}.
 The upshot is that we can choose a contact form for the connecting tube such that all tube orbits lie in a contact sphere in the middle of the tube. This gives generators in odd degree $k$ for $k\geq 2n-3$. These generators can be made to have arbitrarily small action. On the other hand, wandering orbits can have any degree, but their action can be made arbitrarily large.

Now let $(M,\xi)$ be a cooriented contact manifold and form the new contact manifold $(M,\xi) \# (S^{2n-1},\xi_L)$. By taking the connected sum region in $S^{2n-1}$ near the zero-section of a page in the above open book, we see that
$$
HC_*( (M,\xi)\#(S^{2n-1},\xi_L))=0.
$$
Indeed, the energy argument from the previous section still applies, because any holomorphic plane still needs to intersect the binding in the model we constructed for $(S^{2n-1},\alpha_L)$.
 This means that finite energy planes cannot even come close to the connecting tube of the connected sum.
 On the other hand, we can also ensure that the closed Reeb orbit $\gamma_0$ has action smaller than all other orbits except for orbits that lie in the middle of the connecting tube.
 To see this, think of $(S^{2n-1},\xi_L)$ as a very small manifold;
 by adjusting the construction from Section~\ref{sec:prelims} we can make the action of the Reeb orbit $\gamma_0$ arbitrarily small. 
The tube orbits have a positive index that is larger than $1=\deg \gamma_0$ (one can have equality in dimension $3$), so we can exclude holomorphic curves going from $\gamma_0$ to other orbits by either action or degree reasons.
 We have therefore the following theorem.
\begin{theorem}
\label{thm:vanishingCH}
Let $(M,\alpha)$ be any contact manifold and let $(S^{2n-1},\alpha_L)$ be the contact sphere with vanishing contact homology constructed in Section~\ref{sec_openbook_construction}. Suppose $(M\# S^{2n-1},\alpha \# \alpha_L)$ has a well defined contact homology. Then $(M\# S^{2n-1},\alpha \# \alpha_L)$ is algebraically overtwisted.
\end{theorem}

As an application we give an alternative proof of the well known theorem that contact homology of overtwisted contact manifolds vanishes. Namely in dimension $3$, the contact sphere $S_{+1}:=(S^3,\alpha_L)$ corresponds to the overtwisted contact structure on $S^3$ with homotopy class $+1$, see~\cite{Stipsicz:openbook_3manifold}. Let $S_{-1}$ denote the overtwisted contact structure on $S^3$ with homotopy class $-1$. 

Now let $M$ be any overtwisted contact manifold and consider $M\# S_{-1} \# S_{+1}$, which has the same homotopy class of plane fields. By Eliashberg's classification of overtwisted contact structures \cite{Eliashberg:overtwisted} the contact manifold $M\# S_{-1} \# S_{+1}$ is contactomorphic to $M$. By the above theorem, the former manifold has vanishing contact homology.

\section{Connected sums and stabilizations}
\label{sec:stabilization}
Let us begin by defining the plumbing or $2$-Murasugi sum of contact open books.
 Note that it is currently not known whether this operation can be performed in a such a way that the resulting open book is again a contact open book.
\begin{construction}
Let $(M_1,\alpha_1)$ and $(M_2,\alpha_2)$ be contact manifolds.
 Suppose the open book formed by the page $\Sigma_i$ and monodromy $\psi_i$ is a compatible open book for the contact manifold $(M_i,\alpha_i)$.
 Suppose that $L_i$ is a properly embedded Lagrangian ball with Legendrian boundary in $\Sigma_i$ for $i=1,2$.
 Then the {\bf plumbing} $P(\Sigma_1,\Sigma_2;L_1,L_2)$ of the pages is defined by taking the plumbing of $\Sigma_1$ and $\Sigma_2$ along $L_1$ and $L_2$. 
More precisely, by the Weinstein neighborhood theorem we get standard neighborhoods of $L_i$ which are symplectomorphic to $(T^*D^n,\lambda_{can})$. If we use coordinates $(q,p)$ for an element in $T^*D^n$, then $\lambda_{can}=pdq$. 
These coordinates can be used for the plumbing. 
We identify the $q$-coordinates of $L_1$ with the $p$-coordinates of $L_2$ and vice versa.

Now extend $\psi_j$ to $\tilde \psi_j$ by requiring these maps to be the identity outside the domain of $\psi_j$.  
We obtain an open book by taking the monodromy $\tilde \psi_2 \circ \tilde \psi_1$.
\end{construction}
Morally speaking, if $\Sigma_1$ and $\Sigma_2$ are Stein, then we expect that $P(\Sigma_1,\Sigma_2;L_1,L_2)$ is Stein. A simple idea that works for some examples is to interpolate plurisubharmonic functions on both parts of the plumbing. If we suppose that the page is symplectic and $\psi_2\circ \psi_1$ is a symplectomorphism, then it is not clear that $\psi_2\circ \psi_1$ is an exact symplectomorphism.
 However, a simple trick of Giroux shows that we can then always construct a contact open book by deforming the symplectomorphism.
\begin{remark}
The composition of the monodromies does not necessarily send a point $x\in \Sigma_1$ to a point in $\Sigma_1$.
 Indeed, suppose that $\psi_1$ sends $x$ to the plumbing ball.
 Then $\psi_2$ might send that point outside the plumbing ball and hence outside $\Sigma_1$.
\end{remark}

In dimension $3$ the situation is simpler though.
 The resulting open book with page $P(\Sigma_1,\Sigma_2;L_1,L_2)$ and monodromy $\psi_2 \circ \psi_1$ is an open book supporting a contact structure. In fact, we have the following theorem due to Torisu \cite{Torisu:Murasugi_sum}.
\begin{theorem}[Plumbing or $2$-Murasugi sum]
\label{thm:plumbing_openbook}
The $2$-Murasugi sum of the contact open books $(\Sigma_1,\psi_1)$ and $(\Sigma_2,\psi_2)$ along the arcs $L_1$ and $L_2$ is contactomorphic to the contact connected sum $(\Sigma_1,\psi_1) \# (\Sigma_2,\psi_2) $.
\end{theorem}
Note that in the higher dimensional case, we also have that the plumbing operation of open books is diffeomorphic to the connected sum of the manifolds.

If the Lagrangian submanifolds $L_1$ and $L_2$ which we use for the plumbing are boundary-parallel, then the Murasugi sum amounts to a book connected sum. The latter gets an induced open book compatible with the contact connected sum of $(M_1,\alpha_1)$ and $(M_2,\alpha_2)$. Therefore, the following conjectures seem reasonable and in fact, current work of Giroux on open books (see also~\cite{Giroux:ICM}) should establish these conjectures.
\begin{conjecture}
Suppose we are given contact open books $(\Sigma_1,\psi_1)$ and $(\Sigma_2,\psi_2)$ with properly embedded Lagrangian balls $L_1\subset \Sigma_1$ and $L_2\subset \Sigma_2$.
 Assume that these Lagrangian balls have Legendrian boundary.
 Then there is are deformations $\tilde \psi_i$ of $\psi_i$ that are isotopic as symplectomorphisms to $\psi_i$ with the following property.
 The open book $(P(\Sigma_1,\Sigma_2;L_1,L_2),\tilde \psi_2 \circ \tilde \psi_1)$ supports the contact structure on the connected sum $(M_1,\alpha_1)\#(M_2,\alpha_2)$.
\end{conjecture}
Note that this conjecture would imply that a right-handed stabilization of a contact manifold $(M,\xi)$ is contactomorphic to $(M,\xi)$. Indeed, first observe that the standard structure $\xi_0$ on $S^{2n+1}$ admits an open book with page $T^*S^n$ and monodromy a right-handed Dehn twist. 
Then a stabilization of $(M,\xi)$ along a Lagrangian $L$ can simply be regarded as the plumbing of $(M,\xi)$ and $(S^{2n+1},\xi_0)$ along $L$ and a fiber of $T^*S^n$.
\begin{conjecture}
Let $(M,\alpha)$ be a cooriented contact manifold with dimension $\geq 3$. Let $(\Sigma,\psi)$ be a supporting open book for $(M,\alpha)$ such that $\Sigma$ contains a Lagrangian ball with boundary a Legendrian sphere inside $\partial \Sigma$. Then the left-handed stabilization of $(M,\alpha)$ along $L$ has vanishing contact homology.
\end{conjecture}
Note that the conjecture can be proved in dimension $3$ by Theorem~\ref{thm:vanishingCH} and Theorem~\ref{thm:plumbing_openbook} and that it also holds true in any dimension provided $L$ is boundary-parallel by direct application of Theorem~\ref{thm:vanishingCH}.

\bibliographystyle{amsplain}
\bibliography{ch_left}

\end{document}